\documentclass{amsart}
\usepackage{graphicx}
\usepackage{amsmath,amssymb}
\usepackage{amsthm}

\newtheorem{thm}{Theorem}[section]
\newtheorem{cor}[thm]{Corollary}
\newtheorem{lem}[thm]{Lemma}
\newtheorem{prop}[thm]{Proposition}

\theoremstyle{definition}

\theoremstyle{remark}
\newtheorem{rem}[thm]{Remark}
\theoremstyle{example}

\theoremstyle{conjecture}

\numberwithin{equation}{section}

\newcommand{\del}{\delta}

\newcommand{\R}{{\mathbb R}}

\newcommand{\Q}{{\mathbb Q}}
\newcommand{\N}{{\mathbb N}}

\newcommand{\XX}{{\mathbb X}}
\newcommand{\YY}{{\mathbb Y}}
\newcommand{\Z}{{\mathbb Z}}

\newcommand{\calA}{{\mathcal A}}

\newcommand{\calD}{{\mathcal D}}

\newcommand{\calM}{{\mathcal M }}

\newcommand{\calS}{{\mathcal S}}
\newcommand{\calT}{{\mathcal T}}

\begin{document}

\title[Maximal truncated Calder\'on-Zygmund operators]{A unified method for maximal truncated Calder\'on-Zygmund operators in general function spaces by sparse domination}%

\author{Theresa C. Anderson and Bingyang Hu}

\address{Theresa C. Anderson: Department of Mathematics, Purdue University, 150 N. University St., West Lafayette, IN 47907, U.S.A.}%
\email{tcanderson@purdue.edu}

\address{Bingyang Hu: Department of Mathematics, University of Wisconsin, Madison, 480 Lincoln Dr., Madision, WI 53705, U.S.A.}%
\email{bhu32@wisc.edu}

\date{\today}%


\maketitle

{\bf Abstract}. 
In this note we give simple proofs of several results involving maximal truncated Calde\'on-Zygmund operators in the general setting of rearrangement invariant quasi-Banach function spaces by sparse domination.  Our techniques allow us to track the dependence of the constants in weighted norm inequalities; additionally, our results hold in $\mathbb{R}^n$ as well as in many spaces of homogeneous type. 

\bigskip

\section{Introduction}

\bigskip

Sparse domination has been an extremely active area of research recently in harmonic analysis.  This technique dates back to Andrei Lerner from his alternative, simple proof of the $A_2$ theorem \cite{AL1, AL2}, proved originally by Hyt\"onen \cite{TH}.  Lerner is able to bound all \emph{Calder\'on-Zygmund Operators} (CZOs) by a supremum of a special collection of dyadic, positive operators called \emph{sparse operators}.  This bound led almost instantly to a proof of the sharp dependence of the constant in related weighted norm inequalities, the $A_2$ theorem, a problem that had been actively worked on for over a decade.

There have been many improvements to Lerner's techniques as well as to extending his ideas to a wide range of spaces and operators.  We mention a few of these in our references, however, these results are too numerous to list fully; we refer the interested reader to the many recent papers and monographs involving sparse domination for more references and background.   We could have made use of some of these improvements, such as \cite{CR, La, LN, CDO}, but since we are looking at weighted norm inequalities, Lerner's original technology also works.

In this note, we concentrate on several of the results in the paper \cite{CCMP} involving the maximal truncated CZO.  Specifically, we study the behavior of the maximal truncated CZO on rearrangement invariant Banach function spaces (RIBFS), rearrangement invariant quasi-Banach function spaces (RIQBFS); we also show some modular inequalities.  To bring our results into context, we recall a few definitions.

Let $T$ be a Calder\'on-Zygmund operator in $\R^n$ with standard kernel $K$ satisfying the following size and smoothness conditions
\begin{enumerate}
\item[(1).] $|K(x, y)| \le \frac{c}{|x-y|^n}$, where $x \neq y$;
\item[(2).] There exists $0<\del \le 1$ such that 
$$
|K(x, y)-K(x', y)|+|K(y, x)-K(y, x')| \le c \frac{|x-x'|^\del}{|x-y|^{n+\del}},
$$
where $|x-x'| \le |x-y|/2$ and $c$ is some absolute constant.
\item[(3).] $T$ is bounded on $L^2$.
\end{enumerate}

Given a Calder\'on-Zygmund operator $T$, define its maximal truncated operator by 
$$
T^{**}f(x)=\sup_{0<\varepsilon_1<\varepsilon_2} \left| \int_{\varepsilon_1<|x-y|<\varepsilon_2} K(x, y)f(y)dy \right|.
$$

We say that a weight $w$ belongs to the Muckenhoupt class $A_p, 1<p<\infty$, if for every cube $Q \subset \R^n$,
$$
\left( \frac{1}{|Q|} \int_Q w(x)dx \right) \left( \frac{1}{|Q|} \int_Q w(x)^{1-p'} dx \right)^{p-1} \le [w]_{A_p}.
$$
When $p=1$, $w$ belongs $A_1$ if $Mw(x) \le [w]_{A_1}w(x), a.e$. Moreover, we denote $A_\infty=\bigcup_{p \ge 1} A_p$. In this paper, we use the Fujii-Wilson definition of the Muckenhoupt class $A_\infty$. Namely, the weight $w$ belongs to the class $A_\infty$ if and only if
$$
[w]_{A_\infty}:=\sup_{Q} \frac{1}{w(Q)} \int_Q M(w\chi_Q)<\infty,
$$
where the supremum is taken with respect to all cubes in $\R^n$ whose sides are parallel to the axes. 

In this language, the $A_2$ theorem states that for $w \in A_p, p>1$,  
$$
\|T\|_{L^p(w) \mapsto L^p(w)} \le C(T, p) [w]_{A_p}^{\max\left\{1, \frac{1}{p-1} \right\}}
$$
and the exponent is sharp. We refer the readers to the books \cite{LG1, LG2} for more information. 


By using sparse domination, we show that under certain conditions, the following hold, with explicit dependence of the constant $C$ on the weight $w$ (which is detailed and discussed in the body of this paper):
\begin{enumerate}
\item[(i).] 
$$
\|T^{**} f\|_{\XX(w)} \le C\|f\|_{\XX(w)},
$$ 
where $\XX$ is some RIBFS or RIQBFS (see Theorem \ref{thm01} and \ref{thm02} for a precise statement);
\item[(ii).] 
$$ 
\int_{\R^n} \phi(T^{**}f(x)) w(x)dx \le C \int_{\R^n} \phi(f(x)) w(x)dx,
$$
where $\phi$ is a $N$-function and $w$ is some Muckenhoupt weight.
\end{enumerate}
As stated, additionally we track the dependence of the constants on the weight characteristic and provide some commentary.  In particular, the dependence on the constant that we obtain improves on that in \cite{HP1} in certain cases, even on the space $L^2(w)$ (see remarks following Theorem \ref{thm02}).    

Our approach simplifies the original proof, which is done by using extrapolation \cite{CCMP}. Moreover, by taking the advantage of the sparse domination, we can track the constant $C$ and study its dependence with respect to $w$.  Finally this technique is general enough to hold in many spaces of homogeneous type (SHT).  These are doubling measure spaces equipped with a quasimetric -- more references and a precise definition are contained in \cite{AV}.  For simplicity, we structure our results in $\mathbb{R}^n$, and we indicate throughout the note where additional steps are needed for SHT and what they are; we mention any restrictions on the space when they arise. 

The structure of the paper is as follows: Section 2 provides background, especially concerning RIQBFS, and Section 3 includes our main results, proofs, and remarks.

Throughout this paper, for $a, b \in \R$, $a \lesssim b$ ($a \gtrsim
b$, respectively) means there exists a positive number $C$, which is
independent of $a$ and $b$, such that $a \leq Cb$ ($ a \geq Cb$,
respectively).

\subsection{Acknowledgements}
Anderson's work was supported by NSF DMS-1502464.  The authors thank David Cruz-Uribe for helpful discussions.

\bigskip

\section{Preliminaries}

\bigskip

In this section, we collect several basic facts for RIBFS, RIQBFS and modular inequalities.


\subsection{RIBFS and RIQBFS}


Denote $\calM$ as the set of measurable functions on $(\R^n, dx)$ and $\calM^+$ for the nonnegative ones. A \emph{rearrangement invariant Banach norm} is a mapping $\rho: \calM^+ \mapsto [0, \infty]$ such that the following properties hold:
\begin{enumerate}
\item[(a).] $\rho(f)=0 \Leftrightarrow f=0, a.e.$; $\rho(f+g) \le \rho(f)+\rho(g)$; $\rho(af)=a\rho(f)$ for $a \ge 0$;
\item[(b).] If $0 \le f \le g, a.e.$, then $\rho(f) \le \rho(g)$;
\item[(c).] If $f_n \uparrow f, a.e.$, then $\rho(f_n) \uparrow \rho(f)$;
\item[(d).] If $E$ is a measurable set such that $|E|<\infty$, then $\rho(\chi_E)<\infty$, and $\int_E fdx \le C_E \rho(f)$, for some constant $0<C_E<\infty$, depending on $E$ and $\rho$, but independent of $f$.
\item[(e).] $\rho(f)=\rho(g)$ if $f$ and $g$ are equimeasurable, that is, $d_f(\lambda)=d_g(\lambda), \lambda \ge 0$, where $d_f$ ($d_g$ respectively) denotes the distribution function of $f$ ($g$ respectively). 
\end{enumerate}
By means of $\rho$, the \emph{rearrangement invariant Banach function spaces (RIBFS)} is defined as
$$
\XX=\left\{f \in \calM: \|f\|_\XX:=\rho(|f|)<\infty \right\}.
$$
Moreover, the \emph{associate space of $\XX$} is the Banach function $\XX'$ defined by 
$$
\XX'=\left\{f \in \calM, \|f\|_{\XX'}=\sup \left\{ \int_{\R^n} fgdx: g \in \calM^+, \rho(g) \le 1 \right\}<\infty \right\}.
$$
Note that in the present setting, $\XX$ is an RIBFS if and only if $\XX'$ is an RIBFS  (see, e.g., \cite[Chapter 2, Corollary 4.4]{BS}).  For SHT, we require that the underlying space be \emph{resonant} (that is, a $\sigma$-finite space that is completely non-atomic, or is atomic with all atoms having equal measure).

An important feature for these spaces is the Lorentz-Luxemburg theorem, which asserts that $\XX=\XX''$ and hence we have
$$
\|f\|_\XX=\sup \left\{ \left| \int_{\R^n} f g dx \right|: g \in \XX', \|g\|_{\XX'} \le 1 \right\}.
$$
Recall the decreasing rearrangement of $f$ is the function $f^*$ on $[0, \infty)$ defined by 
$$
f^*(t)=\inf\left\{ \lambda \ge 0: d_f(\lambda) \le t \right\}, \ t \ge 0.
$$
It is well-known that $f^*$ is equimeasurable with $f$ and hence by Luxemburg's representation theorem, there exists a RIBFS $\overline{\XX}$ over $(\R^+, dx)$, such that $f  \in \XX$ if and only if $f^* \in \overline{\XX}$ with $\|f\|_{\XX}=\|f^*\|_{\overline{\XX}}$, that is, the mapping $f \mapsto f^*$ is an isometry. Furthermore, for associate space, we have
$\overline{\XX}'=\overline{\XX'}$ and $\|f\|_{\XX'}=\|f^*\|_{\overline{\XX}'}$. We refer the readers to the book \cite{BS} for a detailed introduction to RIBFS.

Let $w \in A_\infty$, $\XX$ a RIBFS and $\overline{\XX}$ as its corresponding space in $(\R^+, dx)$. We consider the weighted version of the space $\XX$ as follows:
$$
\XX(w)=\left\{ f \in \calM: \|f\|_{\XX(w)}:=\|f_w^*\|_{\overline{\XX}}<\infty \right\},
$$
where $f_w^*(t)=\inf\left\{ \lambda \ge 0: w_f(\lambda) \le t \right\}, t \ge 0$ is the decreasing rearrangement induced by $w_f$, the distribution function of $f$ with respect to the measure $wdx$ (note that we need a resonant space to apply the representation theorem). It is known that $\XX'(w)=\XX(w)'$ (see \cite{CCMP}).

Next, we recall \emph{Boyd indices} of a RIBFS, which are closely related to some interpolation properties (see, \cite[Boyd's Theorem]{BS}). Consider the dilation operator
$$
D_t f(s)=f\left( \frac{s}{t} \right), \ 0<t<\infty, f \in \overline{\XX},
$$
with the norm
$$
h_\XX(t)=\|D_t\|_{\overline{\XX} \mapsto \overline{\XX}},  \ 0<t<\infty.
$$
Then, the \emph{lower and upper Boyd indices} are defined, respectively, by
$$
p_\XX=\lim_{t \to \infty} \frac{\log t}{\log h_\XX(t)}=\sup_{1<t<\infty} \frac{\log t}{\log h_\XX(t)}, \quad q_\XX=\lim_{t \to 0^{+}} \frac{\log t}{\log h_\XX(t)}=\inf_{0<t<1} \frac{\log t}{\log h_\XX(t)}.
$$
We have that $1 \le p_\XX \le q_\XX \le \infty$, which follows from the fact that $h_\XX(t)$ is submultiplicative, that is, $h_\XX(ts) \le h_\XX(t) h_\XX(s), \forall s, t>0$. The relationship between the Boyd indices of $\XX$ and $\XX'$ is th following: 
$p_{\XX'}=(q_\XX)'$ and $q_{\XX'}=(p_\XX)'$, where $p$ and $p'$ are conjugate exponents. (see, e.g., \cite{BS, LM}).

For each $0<r<\infty$ and $\XX$ a RIBFS, we consider the $r$ exponent of $\XX$. Namely, 
$$
\XX^r=\left\{f \in \calM: |f|^r \in \XX\right\},
$$
with the norm $\|f\|_{\XX^r}=\| |f|^r \|_\XX^{\frac{1}{r}}$. Note that the definition of Boyd indices extends to $\XX^r$. Namely, we have $p_{\XX^r}=p_\XX \cdot r$ and $q_{\XX^r}=q_\XX \cdot r$. It is known that if $\XX$ is a RIBFS and $r \ge 1$, then $\XX^r$ is still a RIBFS, however, for $0<r<1$, the space $\XX^r$ is not necessarily Banach (see, e.g., \cite{CCMP}). Hence, it is natural to consider the quasi-Banach case. 

We start with the definition of the quasi-Banach function norm. Again, let $\rho': \calM^+ \mapsto [0, \infty)$ be a mapping. We say $\rho'$ is a \emph{rearrangement invariant quasi-Banach function norm} if $\rho$ satisfies the defining condition (a), (b), (c), (e) with the triangle inequality replaced by 
$$
\rho'(f+g) \le C(\rho'(f)+\rho'(g)),
$$
where $C$ is an absolute constant. Then, similarly, the \emph{rearrangement invariant quasi-Banach function spaces (RIQBFS)} is defined as the collection of all measurable functions such that $\rho'(|f|)<\infty$. In addition, for the purpose of making $\XX^r$ become a RIBFS for some large power $r$, where $\XX$ is some RIQBFS, we impose the following \emph{$p$-convex} condition on $\XX$ for $p>0$ (see, e.g., \cite{GK}) by requiring
$$
\left\| \left( \sum_{j=1}^N |f_j|^p \right) ^{\frac{1}{p}} \right\|_\XX \lesssim \left( \sum_{j=1}^N \|f_j\|_\XX^p \right)^{\frac{1}{p}}.
$$
Clearly, the $p$-convexity condition is equivalent to the fact that $\XX^{\frac{1}{p}}$ is an RIBFS and again by Lorentz-Luxemburg's theorem, we have
$$
\|f\|_\XX \simeq \sup \left\{ \left( \int_{\R^n} |f(x)|^p g(x)dx \right)^{\frac{1}{p}}: g \in \calM^+, \|g\|_{\YY'} \le 1 \right\},
$$
where $\YY'$ is the associate space of the RIBFS $\YY=\XX^{\frac{1}{p}}$. In a similar fashion, by using the fact that powers commutes with $f^*$, we can define $\XX(w)$ for a RIQBFS $\XX$, $w \in A_\infty$ and $0<r<\infty$, and we have $\XX(w)^r=\XX^r(w)$.  

\begin{rem}
We list some typical examples of RIBFS and RIQBFS here: the Lebesgue space $L^p$, the Lorentz space $L^{p, q}$, the Orlicz spaces $L^\phi$, the Lorentz $\Gamma$-spaces $\Gamma^q(v)$ and the Marcinkiewicz spaces $\mathbb M_\varphi$. We refer the readers to the work \cite{CCMP} for a detailed introduction to these spaces, as well as their Boyd indices.
\end{rem}


\subsection{Modular inequality}

To set up our modular inequality results, we start recalling some basic properties of Young functions, as well as $N$-functions. Let $\Phi$ be the collection of all the functions $\phi: [0, \infty) \mapsto [0, \infty)$ satisfy the following conditions:
\begin{enumerate}
\item[(1).] $\phi$ is non-negative and increasing;
\item[(2).] $\phi(0^+)=0$ and $\phi(\infty)=\infty$.
\end{enumerate}
If $\phi \in \Phi$ is convex, then we say $\phi$ is a \emph{Young function}. Moreover, an \emph{$N$-function} $\phi$ is a Young function such that
$$
\lim_{t \to 0^+} \frac{\phi(t)}{t}=0 \quad \textrm{and} \quad \lim_{t \to \infty} \frac{\phi(t)}{t}=\infty.
$$
We say that $\phi \in \Phi$ is \emph{quasi-convex} if there exists a convex function $\widetilde{\phi}$ and $a_1 \ge 1$ such that 
$$
\widetilde{\phi}(t) \le \phi(t) \le a_1\widetilde{\phi}(a_1t), \quad t \ge 0.
$$
For a positive increasing function $\phi$, we define the lower and upper dilation indices of $\phi$, respectively, by
$$
i_\phi=\lim_{t \to 0^+} \frac{\log h_\phi(t)}{\log t}=\sup_{0<t<1} \frac{\log h_\phi(t)}{\log t}, \quad I_\phi=\lim_{t \to \infty}  \frac{\log h_\phi(t)}{\log t}=\inf_{1<t<\infty} \frac{\log h_\phi(t)}{\log t},
$$
where 
$$
h_\phi(t)=\sup_{s>0} \frac{\phi(st)}{\phi(s)}, \quad t>0.
$$
Observe that $0 \le i_\phi \le I_\phi \le \infty$. Moreover, as we mentioned before, the dilation indices are closely related to Boyd indices. More precisely, we have
$$
p_\XX=\frac{1}{I_\phi}, \quad q_\XX=\frac{1}{i_\phi},
$$
where $\XX$ is the Marcinkiewicz space induced by $\phi$. (see \cite{CCMP}), while
$$
p_\XX=i_\phi, \quad q_\XX=I_\phi
$$
where $\XX$ is the Orlicz space induced by $\phi$. (see \cite{CMP}). 

The following $\Delta_2$ condition is crucial. Given a function $\phi \in \Phi$, we say that $\phi$ satisfies the $\Delta_2$ condition if $\phi$ is doubling, that is, 
$$
\phi(2t) \le C \phi(t), \ t>0.
$$
It is well-known that if $\phi$ is quasi-convex, then $i_\phi \ge 1$, $\phi \in \Delta_2$ if and only if $I_\phi<\infty$ and $\overline{\phi} \in \Delta_2$ if and onlly if $i_\phi>1$, where $\overline{\phi}(s)=\sup_{t>0} \left\{st-\phi(t) \right\}, s>0$ is the \emph{complementary function} of $\phi$. (see, e.g., \cite{RR}). Here are some main properties of $\overline{\phi}$.

\begin{enumerate}
\item[1.] (Young's inequality) $st \le \phi(s)+\overline{\phi}(t), s, t \ge 0$;
\item[2.] When $\phi$ is an $N$-function, then $\overline{\phi}$ is also a $N$-function, and the following inequality holds:
\begin{equation} \label{eq006}
t \le \phi^{-1}(t) \overline{\phi}^{-1}(t) \le 2t,  \ t \ge 0;
\end{equation}
\item[3.] If $\phi$ is an $N$-function, then there exists $0<\alpha<1$ such that $\phi^\alpha$ is quasi-convex if and only if $\overline{\phi} \in \Delta_2$, where $\phi^\alpha(t)=\phi(t)^\alpha$.
\end{enumerate}

Now we are ready to define the modular inequality. Given $w \in A_\infty$ and $\phi \in \Phi$, we define the modular
$$
\rho_w^\phi(f)=\int_{\R^n} \phi(|f(x)|)w(x)dx.
$$
The collection of functions
$$
M_w^\phi=\left\{f: \rho_w^\phi(f)<\infty\right\}
$$
is referred to as a modular space. A sublinear operator $T$ satisfies a modular inequality on $M_w^\phi$ if  there exists constant $c_1, c_2>0$ such that 
$$
\rho_w^\phi(Tf) \le c_1 \rho_w^\phi(c_2 f)
$$ 
and satisfies a weak modular inequality on $M_w^\phi$ if there exists $c_3, c_4>0$ such that
$$
\sup_{\lambda} \phi(\lambda) w\{x \in \R^n: f(x)>\lambda\} \le c_3 \sup_{\lambda} \phi(\lambda) w\{x \in \R^n: c_4g(x)>\lambda\}.
$$
Note that weighted modular estimates are not necessarily associated with Banach or quasi-Banach spaces and so duality cannot be used.  Modular inequalities were originally developed as a means for providing endpoint estimates for certain operators, such as iterates of the Hardy-Littlewood maximal function \cite{CCMP}. 

\bigskip

\section{Main result}

\bigskip

We need some dyadic calculus from \cite{AL1, AL2}. By a \emph{dyadic grid $\calD$}, we mean a collection of cubes with the following properties:
\begin{enumerate}
\item[(i).] For any $Q \in \calD$ its sidelength $\ell_Q$ is of the form $2^k, k \in \Z$;
\item[(ii).] $Q \cap R \in \{Q, R, \emptyset\}$ for any $Q, R \in \calD$;
\item[(iii).] The cubes of a fixed sidelength $2^k$ form a partition of $\R^n$.
\end{enumerate}
An important property for a dyadic grid is the Three Lattice Theorem. It asserts that there are $3^n$ dyadic grids $\calD_\alpha$ such that for any cube $Q \subset R^n$ there exists a cube $Q_\alpha \in \calD_\alpha$ such that $Q \subset Q_\alpha$ and $\ell_{Q_\alpha} \le c_nl_Q$. Moreover, in \cite{JC}, the author showed that the optimal number of the dyadic grids is $n+1$ (see \cite{HP}, \cite{JC}, \cite{AL1} for a discussion). 

We say $\calS \subset \calD$ is a \emph{sparse family} of cubes if for every $Q \in \calS$, 
$$
\left| \bigcup_{P \in \calS, P \subsetneq Q} P \right| \le \frac{1}{2}|Q|.
$$
Equivalently, if we define
$$
E(Q)=Q \backslash \bigcup_{P \in \calS, P \subsetneq Q} P,
$$
then the sets $E(Q)$ are pairwise disjoint and $|E(Q)| \ge \frac{1}{2}|Q|$. Note that in general, the constant $\frac{1}{2}$ in the above definition can be replaced by any $\gamma \in (0, 1)$. However, we will use $\frac{1}{2}$ for simplicity.  Note that the concept of dyadic grid has been well-studied in SHT, as well as the analogue of the Three lattice theorem (called Mei's theorem) (see \cite{AV}, \cite{HK}, \cite{LN}, to name a few). 

Given a dyadic grid $\calD$ and a sparse family $\calS \subset \calD$, we define the \emph{dyadic positive operator} $\calA$ by
$$
\calA f(x)=\calA_{\calD, \calS} f(x)=\sum_{Q \in \calS} f_{Q} \chi_{Q}(x),
$$
where $f_Q=\frac{1}{|Q|} \int_Q f$. Moreover, given a measurable function $f$ on $\R^n$ and a cube $Q$, we define the \emph{median value} of $f$ over $Q$ by
$$
m_f(Q):=\sup\left\{\lambda: \max\left\{|\left\{ x \in Q: f(x)>\lambda\right\}|, |\left\{x \in \Q: f(x)<\lambda \right\}|\right\} \le |Q|/2\right\}.
$$
An important property of this quantity is the following: if $f \in L^1$, then $|m_Q(T^{**}f)| \to 0$ as $|Q| \to \infty$. Indeed, by the proof of \cite[Lemma 5.1]{AV}, we see that
$$
|m_{Q}(T^{**}f)| \le \frac{\|T^{**}f\|_{L^{1, \infty}(Q)}}{|Q|} \le \|T^{**}\|_{1, \infty} \frac{\|f\|_{L^1}}{|Q|},
$$
where it is well-known that $\|T^{**}\|_{1, \infty}<\infty$ (see, e.g., \cite[Theorem 4.2.4]{LG2}).  In SHT this is true as well, as long as $\mu(X) = \infty$, using the weak bound for $T^{**}$ from \cite{HYY} (note that they impose the H\"ormander condition on their operator).
  
Finally, given any $a>0$ and $Q$ a cube, we denote $aQ$ as the cube with the same center of $Q$ and sidelength $a\ell_Q$. 

The following theorem is crucial.

\begin{thm} \label{ALsparse}
Let $T$ be a Calder\'on-Zygmund operator in $\R^n$ with standard kernel $K$ (see the introduction) and $\calD$ a dyadic grid. Then the following assertions hold:
\begin{enumerate}
\item[(1).] Let $f$ be any measurable function on $\R^n$. For any $Q_0 \in \calD$, there exists a sparse family $\calS \subset \calD$ such that for a.e. $x \in Q_0$, 
\begin{equation} \label{eq001}
|T^{**}f(x)-m_{Q_0}(T^{**}f)| \lesssim Mf(x)+\sum_{m=0}^\infty \frac{1}{2^{m\delta}} \calT_{\calS, m} |f|(x), 
\end{equation}
where $M$ is the Hardy-Littlewood maximal operator and 
$$
\calT_{\calS, m}f(x)=\sum_{Q \in \calS} f_{2^m Q} \chi_{Q}(x), m \in \N.
$$
\item[(2).] Let $X$ be a Banach function space, that is, the very last condition (e) of RIBFS is not required. Then
\begin{equation} \label{eq002}
\|Mf\|_\XX \lesssim \sup_{\calD, \calS} \|A_{\calD, \calS} f\|_\XX, \quad f \ge 0,
\end{equation}
and for any $m \in \N$, 
\begin{equation} \label{eq003}
\sup_{\calS \in \calD} \|\calT_{\calS, m}f\|_\XX \lesssim m \sup_{\calD, \calS} \|\calA_{\calD, \calS}f\|_\XX, \quad f \ge 0.
\end{equation}
\end{enumerate}
In particular, we have for any Banach function space $\XX$, 
\begin{equation} \label{eq004}
\|T^{**}f\|_\XX \lesssim \sup_{\calD, \calS} \|\calA_{\calD, \calS} |f|\|_\XX.
\end{equation}
\end{thm}
(see \cite{AL1}).  This result holds in SHT by following the proofs in \cite{AV}, \cite{AL1} for $T$, but substituting the sublinearity of $T^{**}$ for the linearity of $T$.


\subsection{Maximal truncated Calder\'on-Zygmund operator on RIBFS and RIQBFS}


We start by considering the behavior of $T^{**}$ on RIBFS and RIQBFS.

\begin{lem} \label{PerezSIW}
Let $1<p<\infty$ and let $w \in A_p$. Then $w \in A_{p-\varepsilon}$ where 
$$
\varepsilon=\frac{p-1}{1+2^{n+1}[\sigma]_{A_\infty}}
$$
where $\sigma=w^{1-p'}$ is the dual weight. Furthermore
$$
[w]_{A_{p-\varepsilon}} \le 2^{p-1}[w]_{A_p}.
$$
\end{lem}
(see \cite[Corollary 1.1.1 and Lemma 1.1.3]{CP}).

Note that a version of this lemma is true in SHT, see \cite{HPR}.

\begin{lem} \label{Perez2.3}
Let $\XX$ be a RIQBFS which is $p$-convex for some $0<p \le 1$, then if $1<p_\XX \le \infty$, then $M$ is bounded on $\XX(w)$ for all $w \in A_{p_\XX}$. Moreover, when $1<p_\XX<\infty$, we have
$$
\|M\|_{\XX(w) \mapsto \XX(w)} \le  C[w]_{A_{p_\XX}}^{\frac{1}{p_{\XX}}}, 
$$
where $C$ is an absolute constant only depending on $p_\XX$ and $n$. 
\end{lem}

\begin{proof}
The proof of this lemma is contained in the proof of \cite[Theorem 2.3]{CCMP}. Moreover, the upper bound of $\|M\|_{\XX(w) \mapsto \XX(w)}$ comes from tracking the constant by using Lemma \ref{PerezSIW}. 
\end{proof}

We also need the weighted dyadic Hardy-Littlewood maximal operator $M^\calD_w$, given by 
$$
M^\calD_w f(x)=\sup_{x \in Q, Q \in \calD} \frac{1}{w(Q)} \int_Q |f(y)|w(y)dy, \quad f \in L^1_{\textrm{loc}}(\R^n),
$$
where $w \in A_\infty$ and $\calD$ is the given dyadic grid. It is well-known that $M^\calD_w$ maps $L^p(w)$ strongly to $L^p(w)$ for $1<p \le \infty$ and $L^1(w)$ weakly to $L^{1, \infty}(w)$. (see  \cite[Theorem 7.1.9]{LG1} or \cite{LN}). We have the following result for the dyadic maximal function that can be obtained in a similar manner as [8, The 3.2].  Note that this result is independent of the weight characteristic since we are using the dyadic maximal function.

\begin{lem} \label{Perez4.2}
Let $\XX$ be a RIQBFS which is $p$-convex for some $0<p \le 1$. If $p_\XX>1$ and $w \in A_\infty$, then $M^\calD_w$ is bounded on $\XX(w)$. More precisely, we have $\|M^\calD_w\| \le C$, where the absolute constant $C$ only depends on $p_\XX$ and $n$.
\end{lem}

We first deal with the case when $\XX$ is a RIBFS.

\begin{thm} \label{thm01}
Let $T$ be a Calder\'on-Zygmund operator with standard kernel $K$. Let further, $\XX$ be a RIBFS and $w \in A_{p_\XX}$. Then if $1<p_\XX \leq q_\XX <\infty$, then 
$$
\|T^{**} f\|_{\XX(w)} \le  C[w]_{A_\infty} [w]_{A_{p_\XX}}^{\frac{1}{p_\XX}} \|f\|_{\XX(w)},
$$
where $C$ is an absolute constant only depending on $p_\XX$ and $n$. 
\end{thm}

\begin{proof}
First we note that $p_{\XX'}=(q_\XX)'=\frac{q_\XX}{q_\XX-1}>1$, which follows from the fact that $1<q_\XX<\infty$. 

By \eqref{eq004}, it suffices to show that for any $\calD$ a dyadic grid and $\calS \in \calD$ a sparse family, we have
$$
\|\calA_{\calD, \calS} |f|\|_{\XX(w)} \lesssim [w]_{A_\infty} [w]_{A_{p_\XX}}^{\frac{1}{p_\XX}}  \|f\|_{\XX(w)}. 
$$
Indeed, for any $\|h\|_{\XX'(w)} \le 1$ and $Q$ a dyadic cube, put 
$$
h_{Q, w}=\frac{1}{w(Q)} \int_Q h(x)w(x)dx
$$
and then by Lemma \ref{Perez2.3} and Lemma \ref{Perez4.2}, we have
\begin{eqnarray*}
&&\int_{\R^n} \calA_{\calD, \calS}|f|(x)h(x)w(x)dx = \int_{\R^n}\left( \sum_{Q \in \calS} |f|_{Q} \chi_{Q}(x) \right) h(x) w(x)dx\\
&&=\sum_{Q \in \calS} f_{Q} \cdot h_{Q, w} \cdot w(Q)\\
&&\le \sum_{Q \in \calS} \left( \frac{1}{w(Q)} \int_{Q} (Mf(x))^{\frac{1}{2}} (M^\calD_w h(x))^{\frac{1}{2}} w(x)dx \right)^2 w(Q)\\
&& \le 8[w]_{A_\infty} \int_{\R^n} Mf(x) M^\calD_wh(x) w(x)dx \\ 
&& \le 8[w]_{A_\infty} \|Mf\|_{\XX(w)} \|M^\calD_w h\|_{\XX'(w)} \le [w]_{A_\infty} [w]_{A_{p_\XX}}^{\frac{1}{p_\XX}} \|f\|_{\XX(w)} \|h\|_{\XX'(w)} \\
&& \le 8[w]_{A_\infty} [w]_{A_{p_\XX}}^{\frac{1}{p_\XX}} \|f\|_{\XX(w)},
\end{eqnarray*}
where in the second inequality, we apply the Carleson embedding theorem by noting that the Carleson condition 
\begin{equation} \label{Carlesonembedding}
\sum_{Q \subseteq R} w(Q) \le 2[w]_{A_\infty} w(R)
\end{equation} 
holds for any dyadic cube $R \in \calD$ (see \cite[Lemma 4.1]{HP}).

The desired result hence follows by taking superemum over all dyadic grids $\calD$ and all their sparse families $\calS$.
\end{proof}

The following corollary is straightforward.

\begin{cor}
Let $T$ be a Calder\'on-Zygmund operator with standard kernel $K$. Let further, $\XX$ be an RIQBFS, which is $p$-convex for some $p>0$, and $w \in A_{\frac{p_\XX}{p}}$. Then if $p<p_\XX \leq q_\XX <\infty$, then 
$$
\left\|(T^{**} f)^\frac{1}{p}\right\|_{\XX(w)} \le C[w]^{\frac{1}{p}}_{A_\infty} [w]_{A_{\frac{p_\XX}{p}}}^{\frac{1}{p_\XX}}\left\|f^\frac{1}{p}\right\|_{\XX(w)},
$$
where $C$ is a constant only depending on $p$, $p_\XX$ and $n$. 
\end{cor}

\begin{proof}
This is because $\XX^{\frac{1}{p}}$ is an RIBFS and $p_{\XX^{\frac{1}{p}}}=\frac{p_\XX}{p}$. 
\end{proof}

Next, we deal with the case when $\XX$ is an RIQBFS, which is proved in a different way.

\begin{thm} \label{thm02}
Let $T$ be a Calder\'on-Zygmund operator with standard kernel $K$. Let further, $\XX$ be an RIQBFS, which is $p$-convex for some $0<p \le 1$, and $w \in A_{p_\XX}$. Then if $1<p_\XX \le q_\XX< \infty$, then 
$$
\|T^{**} f\|_{\XX(w)} \le C[w]^{\frac{1}{p}}_{A_\infty} [w]_{A_{p_\XX}}^{\frac{1}{p_\XX}} \|f\|_{\XX(w)},
$$
where $C$ is an absolute constant only depending on $p_\XX$ and $n$. 
\end{thm}

\begin{proof}
Since $\XX$ is $p$-convex, we have that $\YY=\XX^{\frac{1}{p}}$ is an RIBFS. Take and fix any $h \in \YY'(w)$ with $\|h\|_{\YY'(w)} \le 1$. We have the following claim: for any dyadic grid $\calD$ and $\calS \in \calD$ a sparse family, it holds that
$$
I:= \int_{\R^n} \left(Mf(x)+\sum_{m=0}^\infty \frac{1}{2^{m\delta}} \calT_{\calS, m} |f|(x) \right)^p h(x)w(x)dx \le  C[w]_{A_\infty} \left( [w]_{A_{p_\XX}}^{\frac{1}{p_\XX}} \right) ^p\|f^p\|_{\YY(w)}
$$
$$
= C[w]_{A_\infty} \left( [w]_{A_{p_\XX}}^{\frac{1}{p_\XX}} \right) ^p\|f\|_{\XX(w)}^p.
$$
Indeed, we have 
\begin{eqnarray*}
I%
&\le& \int_{\R^n} Mf(x)^ph(x)w(x)dx+ \sum_{m=0}^\infty \int_{\R^n} \frac{1}{2^{m\delta p}} \left[\calT_{\calS, m}|f|(x)\right]^ph(x)w(x)dx\\
&:=&I_1+\sum_{m=0}^\infty  \frac{ I_{2, m}}{2^{m\delta p}}, 
\end{eqnarray*}
where $I_{2, m}:=\int_{\R^n} [\calT_{\calS, m} |f|(x)]^ph(x)w(x)dx$.

\textbf{Estimation of $I_1$.}
By Lemma \ref{Perez2.3} and duality, we have
$$
I_1 \le \|(Mf)^p\|_{\YY(w)}=\|Mf\|_{\XX(w)}^p \le \left( C[w]_{A_{p_\XX}}^{\frac{1}{p_\XX}} \right) ^p \|f\|_{\XX(w)}^p.
$$

\medskip
\textbf{Estimation of $I_{2, m}, m \in \N$.} Using the fact that $0<p \le 1$, the above estimation on $I_1$ and duality, we see that for each $m \in \N$, 
\begin{eqnarray*}
I_{2, m}%
&=& \int_{\R^n} \left( \sum_{Q \in \calS} f_{2^m Q} \chi_{Q}(x) \right)^p h(x)w(x)dx \le \int_{\R^n} \left( \sum_{Q \in \calS} f^p_{2^m Q} \chi_{Q}(x) \right) h(x)w(x)dx \\
&=& \sum_{Q \in \calS} f_{2^mQ}^p \int_{Q}h(x)w(x)dx =\sum_{Q \in \calS} f^p_{2^mQ} h_{Q, w} \cdot w(Q) \\
&\le& \sum_{Q \in \calS}  \left( \frac{1}{w(Q)} \int_{Q} (Mf(x))^{\frac{p}{2}} (M^\calD_w h(x))^{\frac{1}{2}} w(x)dx \right)^2 w(Q)\\
&\lesssim& [w]_{A_{\infty}} \int_{\R^n} (Mf(x))^p M^\calD_wh(x) w(x)dx \\
&& (\textrm{by Carleson embedding theorem})\\
&\le& [w]_{A_\infty} \|(Mf)^p\|_{\YY(w)}\|M^\calD_wh\|_{\YY'(w)} \\
&\le& C[w]_{A_\infty} \left( [w]_{A_{p_\XX}}^{\frac{1}{p_\XX}} \right) ^p \|f\|_{\XX(w)}^p,
\end{eqnarray*}
where in the last inequality, we use the fact that $\|M^\calD_w h\|_{\YY'(w)} \lesssim \|h\|_{\YY'(w)} \le 1$. 

Indeed,  by Lemma \ref{Perez4.2},  it suffices to show $p_{\YY'}>1$. A simple calculation shows that
$$
p_{\YY'}=(q_\YY)'=\frac{q_\YY}{q_\YY-1}=\frac{q_{\XX^{\frac{1}{p}}}}{q_{\XX^{\frac{1}{p}}}-1}=\frac{q_\XX}{q_\XX-p}>1
$$
for $1<q_\XX<\infty$. 
Thus, combining the estimation of both $I_1$ and $I_{2, m}$, we get
\begin{gather} \label{eq005}
I \le C[w]_{A_\infty} \left( [w]_{A_{p_\XX}}^{\frac{1}{p_\XX}} \right) ^p \left(\|f\|_{\XX(w)}^p+ \sum_{m=0}^\infty  \frac{\|f\|_{\XX(w)}^p}{2^{m\delta p}} \right) \\
\le  C[w]_{A_\infty} \left( [w]_{A_{p_\XX}}^{\frac{1}{p_\XX}} \right) ^p \|f\|_{\XX(w)}^p. \nonumber
\end{gather}
Recall that $|m_Q(T^{**}f)| \to 0$ as $|Q| \to \infty$. The desired result follows from \eqref{eq001}, \eqref{eq005} and Lebesgue's domination theorem.
\end{proof}

We make several remarks for the above results.

\begin{rem}
	\begin{enumerate}
		
		\item[(1).] It is clear that Theorem \ref{thm01} is a particular case of Theorem \ref{thm02}. 
		\item[(2).] There is another approach for proving Theorem \ref{thm01} by using extrapolation. We sketch the proof here. First, we have for any $w \in A_p$, where $1<p<\infty$
		$$
		\|T^{**}f\|_{L^p(w)} \lesssim \|f\|_{L^p(w)}.
		$$
		(see \cite[Theorem 7.4.6]{LG1}). Second, we apply the extrapolation theory of RIBFS (see \cite[Theorem 4.10]{CMP}) to the above inequality to get the desired result. However, by doing so, it is not clear how the operator norm of $T^{**}$ depends on the weight $w$. 
		\item[(3).] The modifications for SHT include assuming the Lebesgue differentiation theorem holds.  For a thorough reference on this property and SHT in general, see \cite{AM}.
		\item[(4).]  The dependence on the weight of $[w]_{A_\infty}[w]_{A_{p_\XX}}^{1/{p_\XX}}$ (see Theorem \ref{thm01}) in the constant is indicative of the method of proof - the term $[w]_{A_\infty}$ comes from a Coifman-Fefferman style argument using Carleson embedding while the term $[w]_{A_{p_\XX}}^{1/p_\XX}$ comes from the bound for the maximal function.  By observing the proof and Buckley's proof of the sharp bound for the maximal function, we see that our bound should be sharp in terms of the characteristics.  Therefore, we expect that the dependence on the constants is sharp.
	\end{enumerate}
\end{rem}

\begin{rem}
We note that even when considering the space $L^2(w)$, in certain cases, our constant dependence in Theorem \ref{thm01} improves on the dependence in the work \cite{HP1} (note that the result in \cite{HP1} are for the standard CZO, ours is for the maximal truncated CZO).  In particular, our bound is $$[w]_{A_2}^{1/2}[w]_{A_{\infty}}$$ while Hyt\"onen and P\'erez obtain a bound of \[[w]_{A_2}^{1/2}([w]_{A_\infty}+[w^{-1}]_{A_\infty})^{1/2}.\]
  Let $n=1$.  For the case of power weights $w(x) = |x|^a$ with $0<a<1$, we have that $[w]_{A_2} \approxeq \frac{1}{1+a}\cdot\frac{1}{1-a}$, $[w]_{A_\infty} \approxeq \frac{1}{1+a}$ and $[w^{-1}]_{A_\infty} \approxeq \frac{1}{1-a}$.  Therefore, $$[w]_{A_\infty} \approxeq \frac{1}{1+a} \lesssim (2[w]_{A_\infty}[w^{-1}]_{A_\infty})^{1/2} = ([w]_{A_\infty}+[w^{-1}]_{A_\infty})^{1/2}.$$  Hence in these cases, our bound is smaller (see \cite{HP1} for computations and more details).	
\end{rem}

To close this subsection, we prove the following bound for the median (which holds in SHT), which can be substituted in the above proofs for $\XX(w) = L^p(w)$.  This bound shows the dependence of the constant on the weight characteristic and allows us to consider SHT of finite measure for the case of the RIBFS $L^p(w)$.
\begin{prop}
	We have that $\|m_Q(T^{**})\|_{L^p(w)} \leq C_{T,n} [w]_{A_p}^{1/p}\|f\|_{L^p(w)}$.
\end{prop}
\begin{proof}
	By \cite[Lemma 3.15]{AV} and H\"older's inequality,  we have
	\[
	\|m_Q(T^{**})\|_{L^p(w)} \leq \left(\int_Q \frac{\|T^{**}\|_{1,\infty}^p\|f\|_{L^1(Q)}^p}{|Q|^p}w(x)dx\right) ^{1/p}
	\]
	\[
	\leq \frac{\|T^{**}\|_{1,\infty}}{|Q|}\left( \int_Q f^pw\right) ^{1/p}\left(\int_Q w^{-p'/p}\right) ^{1/p'}\left(\int_Q w(x)dx\right) ^{1/p}
	\]
	\[
	\leq \|T^{**}\|_{1,\infty}\|f\|_{L^p(w)}[w]_{A_p}^{1/p}
	\]
\end{proof}

This bound for the median mirrors the Buckley bound for the maximal function \cite{B}.

\subsection{Maximal truncated Calder\'on-Zygmund operator of modular inequality type}


We need the following lemma, which can be regarded as a modular inequality version of Lemma \ref{Perez2.3}.

\begin{lem} \label{Perez3.7}
Let $\phi \in \Phi$ be such that $\phi$ is quasi-convex. Let further, $w \in A_{i_\phi}$. If $1<i_\phi<\infty$, we have
$$
\int_{\R^n} \phi(Mf(x))w(x)dx \le C_0 \int_{\R^n} \phi\left( C_0 [w]_{A_{i_\phi}}^\frac{1}{i_\phi} |f(x)| \right)w(x)dx\\
$$
where $C_0$ is an absolute constant only depends on $\phi$ and $\alpha$. 
\end{lem}

\begin{proof}
The proof of this lemma is contained in the proof of \cite[Theorem 3.7]{CCMP}. Moreover, the constant follows from Lemma \ref{PerezSIW}. 
\end{proof}

Similarly, as what we did for the RIBFS and RIQBFS case, we need the following version of the modular inequality for the weighted Hardy-Littlewood maximal function.

\begin{lem} \label{Perez5.1}
Let $w \in A_\infty$ and $\phi \in \Phi$ be such that there exists $0<\alpha<1$ for which $\phi^\alpha$ is a quasi-convex function. Then, there exists some constant $a_2>1$, depending on $\phi$ and $w$, such that
$$
\int_{\R^n} \phi(M^\calD_w f(x))w(x)dx \le a_2 \int_{\R^n} \phi(a_2|f(x)|)w(x)dx,
$$
where the constant $a_2$ only depends on $\phi$ and $\alpha$, and is independent of $w$. 
\end{lem}
(see \cite[Propostion 5.1]{CCMP} -- there it is stated that the constant depends on $w$, but by their proof one sees that it is in fact independent of $w$).

\medskip

We make some easy observations of the $\Delta_2$ condition and $N$-functions before we state the following lemma. First, we note that $\phi \in \Delta_2$, that is $\phi(2t) \le C\phi(t), t \ge 0$ if and only if there exists some constant $C'$ (for example, we can take $C'=\frac{\log C}{\log 2}$) such that for any $\lambda \ge 2$, 
\begin{equation} \label{eq007}
\phi(\lambda t) \le 2^{C'} \lambda^{C'} \phi(t),  \ t>0.
\end{equation}
The proof for this claim is straightforward from the definition, and hence we omit it here. Second, since 
$$
\phi^{-1}(t) \overline{\phi}^{-1}(t) \ge t, t \ge0, 
$$
it follows that
$$
t \overline{\phi}^{-1}(\phi(t)) =\phi^{-1} (\phi(t)) \overline{\phi}^{-1}(\phi(t)) \ge \phi(t), t \ge 0,
$$
which implies that
\begin{equation} \label{eq008}
\overline{\phi} \left( \frac{\phi(t)}{t} \right) \le \phi(t), \ t>0.
\end{equation} 

\begin{lem} \label{lem01}
 Let $\phi$ be a $N$-function and $\phi \in \Delta_2$, that is, $I_\phi<\infty$, and $w \in A_{i_\phi}$. If $i_\phi>1$, for each $m \in \N$, any dyadic grid $\calD$ and $\calS \in \calD$ a sparse family, we have
$$
\int_{\R^n} \phi(\calT_{\calS, m}|f|(x))w(x)dx  \le C''[w]^{1+\alpha C'}_{A_\infty}  \int_{\R^n} \phi(Mf(x))w(x)dx,
$$
where $C'$ is defined in \eqref{eq007} and $C''$ is an absolute constant only depending on $\phi$.
\end{lem}

\begin{proof}
Since $\phi$ is a $N$-function, it is clear that the quantity $\phi(\calT_{\calS, m}|f|(x))=0$ when $\calT_{\calS, m}|f|(x)=0$. Hence, in the sequel, we write the function 
$$
\frac{\phi(\calT_{\calS, m}|f|(x))}{\calT_{\calS, m}|f|(x)},
$$ 
which takes its actual value when $\calT_{\calS, m}|f|(x) \neq 0$ and zero when $\calT_{\calS, m}|f|(x)=0$. 

Moreover, since $\phi$ is $\Delta_2$, it follows that there exists some $0<\alpha \le 1$, such that $\overline{\phi}^\alpha$ is quasi-convex, that is, there exists some convex function $\psi$, such that
\begin{equation} \label{eq009}
\psi(t) \le \overline{\phi}^\alpha(t) \le a_3 \psi (a_3t), \ t>0.
\end{equation}
Note that we can always assume that $a_3 \ge 2$. 

 Take and fix some $\varepsilon$, satisfying
\begin{equation} \label{eq010}
0<\varepsilon< \left(\frac{1}{16 a_2[w]_{A_\infty} }\right)^\alpha \cdot \frac{1}{a_2a_3^2}=\min \left\{\frac{1}{2}, \frac{1}{a_2a_3}, \left(\frac{1}{16 a_2[w]_{A_\infty} }\right)^\alpha \cdot \frac{1}{a_2a_3^2} \right\}
\end{equation}
Then by Lemma \ref{Perez5.1}, we have
\begin{eqnarray*}
&&\int_{\R^n} \phi(\calT_{\calS, m}|f|(x))w(x)dx=\int_{\R^n} \left( \sum_{Q \in \calS} f_{2^m Q} \chi_{Q}(x) \right) \cdot \frac{\phi(\calT_{\calS, m}|f|(x))}{\calT_{\calS, m}|f|(x)} w(x)dx\\
&&= \sum_{Q \in \calS} f_{2^mQ} \cdot \frac{1}{w(Q)} \int_{Q} \frac{\phi(\calT_{\calS, m}|f|(x))}{\calT_{\calS, m}|f|(x)} w(x)dx \cdot w(Q)\\
&& \le 8[w]_{A_\infty}  \int_{\R^n} Mf(x) M^\calD_w\left( \frac{\phi(\calT_{\calS, m}|f|(x))}{\calT_{\calS, m}|f|(x)} \right) w(x)dx \\
&& \quad (\textrm{by Carleson embedding theorem}) \\
&&= 8[w]_{A_\infty} \int_{\R^n} \frac{Mf(x)}{\varepsilon} \cdot \varepsilon M^\calD_w \left( \frac{\phi(\calT_{\calS, m}|f|(x))}{\calT_{\calS, m}|f|(x)} \right) w(x)dx \\
&& \le 8[w]_{A_\infty} \left(\int_{\R^n} \phi\left(\frac{Mf(x)}{\varepsilon}\right)w(x)dx+ \int_{\R^n} \overline{\phi} \left(\varepsilon M^\calD_w \left( \frac{\phi(\calT_{\calS, m}|f|(x))}{\calT_{\calS, m}|f|(x)} \right) \right)  w(x)dx \right)\\
&& \quad (\textrm{by \eqref{eq006}}) \\
&& \le \frac{2^{C'+3}[w]_{A_\infty} }{\varepsilon^{C'}}  \int_{\R^n} \phi(Mf(x))w(x)dx \\
&& \quad \quad \quad \quad \quad \quad+ 8[w]_{A_\infty} a_2  \int_{\R^n} \overline{\phi} \left( \frac{a_2 \varepsilon \phi(\calT_{\calS, m}|f|(x))}{\calT_{\calS, m}|f|(x)} \right) w(x)dx \\
&&\le \frac{2^{C'+3}[w]_{A_\infty} }{\varepsilon^{C'}} \int_{\R^n} \phi(Mf(x))w(x)dx\\
&& \quad \quad \quad \quad \quad \quad +8[w]_{A_\infty}a_2 \cdot (a_3^2a_2 \varepsilon)^{\frac{1}{\alpha}} \int_{\R^n} \overline{\phi} \left(\frac{ \phi(\calT_{\calS, m}|f|(x))}{\calT_{\calS, m}|f|(x)} \right) w(x)dx \\
&& \quad (\textrm{by \eqref{eq009}, \eqref{eq010}}) \\
&& \le \frac{2^{C'+3}[w]_{A_\infty} }{\varepsilon^{C'}} \int_{\R^n} \phi(Mf(x))w(x)dx +\frac{1}{2} \int_{\R^n} \phi(\calT_{\calS, m}|f|(x))w(x)dx \\
&& \quad (\textrm{by \eqref{eq008} and \eqref{eq010}}). 
\end{eqnarray*}

where we have used that fact that $\psi(\lambda t) \leq \lambda\psi(t)$ for $a_3a_2\varepsilon = \lambda \in (0,1)$ since $\psi(0) = 0$.
Hence, we have
$$
\int_{\R^n} \phi(\calT_{\calS, m}|f|(x))w(x)dx \le C''[w]^{1+\alpha C'}_{A_\infty}  \int_{\R^n} \phi(Mf(x))w(x)dx,
$$
where $C''$ is an absolute constant only depending on $\phi$. The lemma is proved.
\end{proof}

\begin{thm} \label{thm03}
Let $T$ be a Calder\'on-Zygmund operator with standard kernel $K$. Let further, $\phi$ be a $N$-function belonging to $\Delta_2$, that is, $I_\phi<\infty$, and $w \in A_{i_\phi}$. If $i_\phi>1$, we have
\begin{equation} \label{eq011}
\int_{\R^n} \phi(|T^{**}f(x)|)w(x)dx \le C(\phi, w) \int_{\R^n} \phi(|f(x)|)w(x)dx,
\end{equation}
where
\[ 
C(\phi, w)=\begin{cases}
C'''[w]_{A_{\infty}}^{1+\alpha C'}, & C_0[w]_{{A_{i_\phi}}}^{\frac{1}{i_\phi}}<2;\\
C''' [w]_{A_{\infty}}^{1+\alpha C'} \left([w]_{{A_{i_\phi}}}^{\frac{1}{i_\phi}} \right)^{C'}, & C_0[w]_{{A_{i_\phi}}}^{\frac{1}{i_\phi}} \ge 2,
\end{cases} \]
where $C_0$ is the constant defined in Lemma \ref{Perez3.7}.
\end{thm}

\begin{proof}
We start with the case when $i_\phi>1$. Denote 
$$
2<K_0=1+\sum_{m=0}^\infty \frac{1}{2^{m\del}}<\infty,
$$ 
where $\del$ is the constant in the smoothing condition of the kernel $K$. Again, we prove a similar claim as what we did in Theorem \ref{thm02}: for any dyadic grid $\calD$ and $\calS \in \calD$ a sparse family, it holds that
$$
J:=\int_{\R^n} \phi\left(Mf(x)+\sum_{m=0}^\infty \frac{1}{2^{m\del}} \calT_{\calS, m}|f|(x) \right) w(x)dx \lesssim \int_{\R^n} \phi(|f(x)|)w(x)dx. 
$$
Indeed, we have
\begin{eqnarray*}
J%
&=& \int_{\R^n} \phi\left(K_0 \cdot \left[\frac{Mf(x)}{K_0}+\sum_{m=0}^\infty \frac{1}{2^{m\del}K_0} \calT_{\calS, m}|f|(x)\right] \right) w(x)dx \\
&\le& C'''\int_{\R^n} \phi\left( \frac{Mf(x)}{K_0}+\sum_{m=0}^\infty \frac{1}{2^{m\del}K_0} \calT_{\calS, m}|f|(x) \right) w(x)dx \\
&& \quad (\textrm{by $\Delta_2$ condition})\\
&\le& \frac{C'''}{K_0} \int_{\R^n} \phi(Mf(x))w(x)dx+\sum_{m=0}^\infty \frac{C'''}{2^{m\del}K_0} \int_{\R^n} \phi(\calT_{\calS, m}|f|(x))w(x)dx \\
&& \quad (\textrm{by convexity of $\phi$}) \\
&\le&  C'''[w]_{A_{\infty}}^{1+\alpha C'} \int_{\R^n} \phi(Mf(x))w(x)dx\\
&& (\textrm{by Lemma \ref{lem01}})\\
&\le& C'''[w]_{A_{\infty}}^{1+\alpha C'}  \int_{\R^n} \phi\left( C_0[w]_{{A_{i_\phi}}}^{\frac{1}{i_\phi}}|f(x)| \right)w(x)dx \\
&& (\textrm{by Lemma \ref{Perez3.7}}),
\end{eqnarray*}
We consider two different cases. 

\medskip

\textit{Case I: $C_0[w]_{{A_{i_\phi}}}^{\frac{1}{i_\phi}}<2$.}

In this case, we have
$$
J \le C''' [w]_{A_\infty}^{1+\alpha C'} \int_{\R^n} \phi(|f(x)|)w(x)dx.
$$

\textit{Case II: $C_0[w]_{{A_{i_\phi}}}^{\frac{1}{i_\phi}} \ge 2$.}

\medskip

By equation \eqref{eq007}, we have
$$
J \le  C''' [w]_{A_{\infty}}^{1+\alpha C'} \left([w]_{{A_{i_\phi}}}^{\frac{1}{i_\phi}} \right)^{C'}  \int_{\R^n} \phi(|f(x)|)w(x)dx.
$$

Finally, combining the above estimation with \eqref{eq001} and Lebesgue's domination theorem, we get the desired result. 

\end{proof}

\begin{rem}
Our constant is not predicted to be sharp here. We conjecture that the sharp constant depends on $[w]_{A_\infty}$ linearly.
\end{rem}


\begin{thebibliography}{99}
\bibitem{AM} Alvarado, R., and Mitrea, M. (2015). Hardy Spaces on Ahlfors-Regular Quasi Metric Spaces (Vol. 2142). Cham: Springer International Publishing.
\bibitem{AV} T. Anderson, A. Vagharshakyan,  A simple proof of the sharp weighted estimate for
Calderon-Zygmund operators on homogeneous spaces, \textit{Journal of Geometric Analysis}. July 2014, Volume
{\bf 24}, Issue 3, pp 1276--1297.

\bibitem{BS} C. Bennett, R. Sharply, Interpolation of Operators, Academic Press, New York, 1988.

\bibitem{B}  S.M. Buckley. Estimates for operator norms on weighted spaces and reverse Jensen inequalities. \textit{Trans. Am. Math. Soc.} 340(1), 253–272 (1993)

\bibitem{JC} J. Conde, A note on dyadic coverings and nondoubling Calderón-Zygmund theory, \textit{J. Math. Anal. Appl.} {\bf 397} (2013), no. 2, 785--790. 

\bibitem{CR} J. M. Conde-Alonso and G. Rey, \textit{A pointwise estimate for positive dyadic shifts and some applications}, Mathematische Annalen 365 (2016), 1111-1135.

\bibitem{CDO} A. Culiuc, F. D. Plinio, and Y. Ou, \textit{Uniform sparse domination of singular integrals via dyadic shifts}, preprint, 2016. To appear in Math. Res. Lett.

\bibitem{CCMP} G. Curbera, J. Cuerva, J. Martell, C. P\'erez, Extrapolation with weights, rearrangement-invariant function spaces, modular inequalities and applications to singular integrals, \textit{Adv. In. Math} {\bf 203} (2006), 256--318.

\bibitem{CMP}D. Cruz-Uribe , J. M Martell and C. P\'erez , \textit{Weights, Extrapolation and the Theory of Rubio de Francia} (Operator Theory: Advances and Applications 215 ), Birkhäuser/Springer (Basel, 2011)

\bibitem{HYY} Guoen Hu, Dachun Yang, and Dongyong Yang. Boundedness of maximal singular integral operators on spaces of homogeneous type and its applications. \textit{J. Math. Soc. Japan} 59 (2007), no. 2, 323–349. 

\bibitem{TH} T. Hyt\"onen, The sharp weighted bound for general Calder\'on–Zygmund operators, \textit{Annals of Mathematics} {\bf 175} (2012), 1473--1506.

\bibitem{HK} T. Hyt\"onen and A. Kairema, Systems of dyadic cubes in a doubling metric space, \textit{Colloq. Math.} 126 (2012), no. 1, 1–33.

\bibitem{HP} T. Hyt\"onen, C. P\'erez, The $L(\log L)^\varepsilon$ endpoint estimate for maximal singular integral operators, \textit{J. Math. Anal. Appl}. {\bf 428}(2015), 605--626.


\bibitem{HP1} T. Hyt\"onen, C. P\'erez,, Sharp weighted bounds involving $A_\infty$.  \textit{Anal. PDE} 6 (2013), no. 4, 777–818.

\bibitem{HPR} T. Hyt\"onen, C. P\'erez, E. Rela. Sharp reverse Hölder property for $A_\infty$ weights on spaces of homogeneous type. \textit{J. Funct. Anal.} 263(12), 3883–3899 (2012).

\bibitem{La} Lacey, Michael T., \textit{An elementary proof of the A2 bound}.  Israel J. Math. 217 (2017), no. 1, 181-195. 

\bibitem{LG1} L. Grafakos, \textit{Classical Fourier Analysis, Third Edition}, Graduate Texts in Math., no {\bf 249}, 
Springer, New York, 2014.

\bibitem{LG2} L. Grafakos,  \textit{Modern Fourier Analysis, Third Edition}, Graduate Texts in Math., no {\bf 250}, 
Springer, New York, 2014.

\bibitem{LN}  A. K. Lerner and F. Nazarov, \textit{Intuitive dyadic calculus: the basics}, arXiv:1508.05639.

\bibitem{GK} L. Grafakos, N. Kalton, Some remarks on multilinear maps and interpolation, \textit{Math. Ann.} {\bf 319} (2001), 151--180.

\bibitem{HP} T. P. Hyt\"onen, C. P\'erez, Sharp weighted bounds involving $A_\infty$, \textit{Anal. PDE} {\bf 6} (2013), 777--818.

\bibitem{AL1} A. Lerner, A simple proof of the $A_2$ conjecture, \textit{Int. Math. Res. Not}. 2013. no. {\bf 14}, 3159--3170.

\bibitem{AL2} A. Lerner,  On an estimate of Calderón-Zygmund operators by dyadic positive operators, \textit{J. Anal. Math}. {\bf 121} (2013), 141--161.

\bibitem{LN} A. Lerner. F. Nazarov, \textit{Intuitive dyadic calculus: the basics}, arXiv:1508.05639.

\bibitem{LM}  L. Maligranda, \textit{Orlicz Spaces and Interpolation}, Seminars in Mathematics 5, Univ. Estadual
de Campinas, Campinas SP, Brazil 1989.

\bibitem{CP} C. P\'erez, \textit{Singular Integrals and weights: a modern introduction}, Spring School on Function spaces and Inequalities, Paseky, June 2013.

\bibitem{RR} M. Rao, Z. Ren, \textit{Theory of Orlicz spaces}, Marcel Dekker, 1991.

\end{thebibliography}
\end{document}